\documentclass{amsart}
\usepackage{amsmath,amsthm,amssymb}

\title{On a conjecture of Cigler}

\author{D. Stanton}
\address{School of Mathematics, University of Minnesota, Minneapolis, MN 55455}
\email{stanton@math.umn.edu}
\date{\today}

\thanks{The author was supported by NSF grant DMS-1148634}

\newtheorem{prop}{Proposition}[section]
\newtheorem{defn}[prop]{Definition}
\newtheorem{conj}[prop]{Conjecture}

\newtheorem{cor}[prop]{Corollary}
\newtheorem{theorem}[prop]{Theorem}

\begin{document}

\begin{abstract}
A conjecture of Cigler which realizes normalized 
$q$-Hermite polynomials as moments is verified.
\end{abstract}

\maketitle

\section{Introduction}
The $q$-Hermite polynomials $H_n(x|q)$ may be defined by 
$$
H_n(x|q)=\sum_{k=0}^n \left[\begin{matrix} n\\k\end{matrix}\right]_q e^{-i(n-2k)\theta}, \quad x=\cos\theta.
$$
(We have used the standard notation for $q$-series given in \cite{GaRa} and 
\cite{Ismbook}.)
In \cite[Th. 4.1]{IsS}, four moment representations are given for four different normalizations of these polynomials
$$
H_n(x|q), \quad \frac{H_n(x|q)}{(q;q)_n}, \quad \frac{H_n(x|q^2)}{(q;q)_n}, \quad \frac{H_n(x|q^2)}{(-q;q)_n}.
$$
 Four other moment 
representations are alluded to at the end of Section 4 of \cite{IsS},
$$
\frac{H_{2n}(x|q)}{(q^2;q^2)_n}, 
\quad \frac{H_{2n}(x|q)}{(q;q^2)_n}, \quad \frac{H_{2n+1}(x|q)}{(q^2;q^2)_n}, \quad \frac{H_{2n+1}(x|q)}{(q^3;q^2)_n}.
$$

Cigler \cite[Section 3.4]{Cig} conjectured that the 6th and 8th of these 
moment sequences may be combined into a single moment sequence
$$
\mu_n=P_n(a)=\frac{1}{(q;q^2)_{[(n+1)/2]}}
\sum_{k=0}^n 
\left[\begin{matrix} n\\k\end{matrix}\right]_q a^k, \quad n\ge 0
$$
for a specific set of orthogonal polynomials. 
In this short note we prove Cigler's conjecture.

\section{The conjecture}

First we define a set of monic orthogonal polynomials $s_n(x)$ 
via a three-term recurrence relation. 

\begin{defn} 
\label{maindefn}
Let $s_n(x)$ be monic polynomials in $x$ defined by 
$$
s_{n+1}(x)=(x-b_n)s_n(x)-\lambda_n s_{n-1}(x), \quad n\ge 0, 
\quad s_0(x)=1, \ s_{-1}(x)=0,
$$ 
where
$$
\begin{aligned}
b_n=& -\frac{(1-q)}{(1-q^{2n+1})(1-q^{2n-1})(1+a)}\times \\
& \left(  a\frac{(1-q^{2n-1})(1-q^{n+1})(1-q^n)}{1-q}-q^n
\left( \frac{1-q^{n-1}}{1-q}+q^{n+1}\frac{1-q^n}{1-q}\right)(1+a)^2\right),
{\text{ for $n\ge 0$ even,}}\\
b_n=& \frac{(1-q)}{(1-q^{2n+1})(1-q^{2n-1})(1+a)}\times \\
& \left(  a\frac{(1-q^{2n+1})(1-q^{n-1})(1-q^n)}{1-q}-q^{n+1}
\left( \frac{1-q^{n}}{1-q}+q^{n-2}\frac{1-q^{n+1}}{1-q}\right)(1+a)^2\right),
{\text{ for $n\ge 1$ odd,}}\\
\lambda_n=& \frac{q^n(1+a)^2(1-q^{n-1})(1-q^n)}{(1-q^{2n-1})^2}, {\text{ for $n\ge 1$ even,}} \\
\lambda_n=& -\frac{(a+q^n)(a+q^{n-1})(1+aq^{n-1})(1+aq^n)}{(1+a)^2(1-q^{2n-1})^2}, 
{\text{ for $n\ge 1$ odd.}}
\end{aligned}
$$
\end{defn}
\vskip5pt\noindent
\begin{defn} Let $L$ be the linear functional on polynomials in $x$ given by
$$
L(1)=1, \quad L(s_n(x))=0, \quad n\ge 1.
$$
\end{defn}

Cigler \cite[Section 3.4]{Cig} conjectured
\begin{conj} (Cigler)
\label{Cigconj}
$$
L(x^n)=P_n(a).
$$
\end{conj}

\section{The proof}

To prove Conjecture~\ref{Cigconj}, we give
another basis for polynomials which has an explicit $L$ value.
\begin{prop} 
\label{newmoms}
Cigler's conjecture is equivalent to
$$
\begin{aligned}
L\left(\prod_{i=0}^{n-1}(x^2-a^2q^{2i})\right)=& \frac{(-a;q)_{2n}}{(q;q^2)_n},\\
L\left(x\prod_{i=0}^{n-1}(x^2-a^2q^{2i})\right)=& \frac{(-a;q)_{2n+1}}{(q;q^2)_{n+1}}.
\end{aligned}
$$
\end{prop}
\begin{proof} We show that $L(x^n)=P_n(a)$ implies the $L$ values 
which are given in Proposition~\ref{newmoms}.
Let $\epsilon=0$ for the first case and $\epsilon=1$ for the second case. 
Expanding by the $q$-binomial theorem we see that
$$
\begin{aligned}
L & \left(x^\epsilon \prod_{i=0}^{n-1}(x^2-a^2q^{2i})\right)=
\sum_{k=0}^n \left[\begin{matrix} n\\k\end{matrix}\right]_{q^2}(-1)^k 
a^{2k}q^{2\binom{k}{2}}L(x^{2(n-k)+\epsilon})\\
=& \sum_{k=0}^n \left[\begin{matrix} n\\k\end{matrix}\right]_{q^2}(-1)^k 
a^{2k}q^{2\binom{k}{2}}\frac{1}{(q;q^2)_{n-k+\epsilon}}
\sum_{s=0}^{2(n-k)+\epsilon} 
\left[\begin{matrix} 2(n-k)+\epsilon\\s\end{matrix}\right]_{q}
a^s\\
=& 
(q^2;q^2)_n\sum_{p=0}^{2n+\epsilon} \frac{a^p}{(q;q)_{2n+\epsilon-p}} 
\sum_{k=0}^{[p/2]} 
\frac{(-1)^k q^{2\binom{k}{2}}}{(q^2;q^2)_k (q;q)_{p-2k}} \\
=& \frac{1}{(q;q^2)_{n+\epsilon}}
\sum_{p=0}^{2n+\epsilon} a^p \left[\begin{matrix} 2n+\epsilon\\p\end{matrix}\right]_{q}
q^{\binom{p}{2}}= \frac{(-a;q)_{2n+\epsilon}}{(q;q^2)_{n+\epsilon}}.
\end{aligned}
$$
The $k$-sum was evaluated by a limiting case of the $q$-Vandermonde 
sum, and the $p$-sum evaluated by the $q$-binomial theorem, see 
\cite[Appendix]{GaRa}.
\end{proof}

Because we know the values of $L(s_n(x))$, if we can expand 
$\prod_{i=0}^{n-1}(x^2-a^2q^{2i})$ in terms of $s_n(x)$ we can verify 
Proposition~\ref{newmoms}. The next Proposition accomplishes this task.
\begin{prop} For any non-negative integer $n$,
\label{expansion}
$$
\prod_{i=0}^{n-1}(x^2-a^2q^{2i})= \sum_{k=0}^{2n} a_k^{(n)} s_{2n-k}(x),
$$
where
$$
\begin{aligned}
a_{2k}^{(n)}=&\frac{(-aq^{2n-1};1/q)_{2k}}{(q^{4n-2k-1};1/q^2)_k}
\left[\begin{matrix} n\\k \end{matrix} \right]_{q^2},\\
a_{2k+1}^{(n)}=&(1+a)\frac{(-aq^{2n-1};1/q)_{2k}}{(q^{4n-2k-1};1/q^2)_{k+1}}
\left[\begin{matrix} n\\k+1 \end{matrix} \right]_{q^2} (1-q^{2(k+1)}).
\end{aligned}
$$
\end{prop}
\begin{proof} Use induction on $n$ and the three-term recurrence 
relation in Definition~\ref{maindefn} twice.  We must verify that
$$
\begin{aligned}
a_k^{(n+1)}=&-a^2q^{2n}a_{k-2}^{(n)}+a_{k}^{(n)}\\
&+(b_{2n-k+2}+b_{2n-k+1})a_{k-1}^{(n)}\\
&+(\lambda_{2n-k+3}+b_{2n-k+2}^2+\lambda_{2n-k+2})a_{k-2}^{(n)}\\
&+(b_{2n-k+3}\lambda_{2n-k+3}+\lambda_{2n-k+3}b_{2n-k+2})a_{k-3}^{(n)}\\
&+\lambda_{2n-k+4}\lambda_{2n-k+3}a_{k-4}^{(n)}.
\end{aligned}
$$
Since each of the terms is explicitly known, this tedious computation may be done by 
hand or computer.
\end{proof}

\begin{theorem} Cigler's conjecture is true. 
\end{theorem}
\begin{proof}
Proposition~\ref{expansion} shows that the first part of 
Proposition~\ref{newmoms} holds, 
$$
L\left(\prod_{i=0}^{n-1}(x^2-a^2q^{2i})\right)=
a_{2n}^{(n)}=\frac{(-a;q)_{2n}}{(q;q^2)_n}.
$$
For the second part of 
Proposition~\ref{newmoms}, we find the constant term in the expansion of 
$x\prod_{i=0}^{n-1}(x^2-a^2q^{2i}).$ Using Definition~\ref{maindefn} this is 
$$
\begin{aligned}
L\left(x\prod_{i=0}^{n-1}(x^2-a^2q^{2i})\right)=
&a_{2n}^{(n)}b_0+a_{2n-1}^{(n)}\lambda_1\\
=& 
\frac{(-a;q)_{2n+1}}{(q;q^2)_{n+1}}.
\end{aligned}
$$ 
\end{proof}

The Hankel determinant of these moments \cite[Conjecture 3.1 (3.4)]{Cig}  
is thereby evaluated.

\begin{cor} The Hankel determinant is
$$
det(P_{i+j}(a))_{0\le i,j\le n} =\prod_{i=1}^n \lambda_i^{n+1-i}.
$$
where $\lambda_i$ is given in Definition~\ref{maindefn}.
\end{cor}

Additional sets of classical orthogonal polynomials as moments are given in 
\cite{IsS}, \cite{IsS2} and \cite{IsS3}.


\begin{thebibliography}{10}

\bibitem{Cig}
J. Cigler,
\newblock Elementary observations on Rogers-Szeg\"o polynomials, arxiv preprint,
{\tt {1602.07850v1.pdf}}, February, 2016.

\bibitem{GaRa}
G. Gasper and M. Rahman, 
\newblock {\it{Basic Hypergeometric Series}}, 
\newblock second edition  Cambridge University
Press, Cambridge, 2004.

\bibitem{Ismbook} M. E. H. Ismail, {\it Classical and 
Quantum Orthogonal Polynomials in one variable}, Cambridge University Press, 
paperback edition, Cambridge, 2009. 

\bibitem{IsS}
M. E. H. Ismail and D. Stanton,
\newblock $q$-Integral and moment representations
for $q$-orthogonal polynomials,
\newblock {\em Can. J. Math.}, 54 (2002), p. 709-735.

\bibitem{IsS2}
M. E. H. Ismail and D. Stanton, 
Classical orthogonal polynomials as moments, 
Canad. J. Math. 49 (1997), no. 3, p. 520-542.

\bibitem{IsS3}
M. E. H. Ismail and D. Stanton, 
More orthogonal polynomials as moments, Mathematical essays in 
honor of Gian-Carlo Rota (Cambridge, MA, 1996), p. 377-396, 
Progr. Math., 161, Birkh\"auser Boston, Boston, MA, 1998. 



\end{thebibliography}
\end{document}